\documentclass[12pt,a4paper]{article}

\usepackage[margin=2.54cm]{geometry}

\def\Dj{\hbox{D\kern-.73em\raise.30ex\hbox{-}
\raise-.30ex\hbox{}}}
\def\dj{\hbox{d\kern-.33em\raise.80ex\hbox{-}
\raise-.80ex\hbox{\kern-.40em}}}
\usepackage{graphicx}
\usepackage{epstopdf}
\usepackage{amsmath,amsthm,amsfonts,amssymb,amscd,cite}
\allowdisplaybreaks

\usepackage{amsmath, amsthm, amscd, amsfonts, amssymb, graphicx, color}
\usepackage[bookmarksnumbered, plainpages]{hyperref}
\usepackage{pgf,tikz}
\usetikzlibrary{arrows}
\usepackage{diagbox}
\usepackage[margin=1cm,%
            font=small,%
            format=hang,%
            labelsep=period,%
            labelfont=bf]{caption}
\DeclareMathOperator{\Tr}{tr}

\newtheorem{theorem}{Theorem}[section]

\newtheorem{lemma}[theorem]{Lemma}
\newtheorem{corollary}[theorem]{Corollary}

\newtheorem{con}[theorem]{Conjecture}

\begin{document}

\baselineskip=0.30in

\vspace*{25mm}

 \begin{center}
 {\Large \bf  Laplacian Coefficients of a Forest in terms of the Number of Closed Walks in the Forest and its Line Graph}

 \vspace{6mm}

 { \bf Ali Ghalavand$^{1,}$\footnote{Corresponding author.} and Ali Reza Ashrafi$^1$}

 \vspace{3mm}

 \baselineskip=0.20in

 $^1${\it Department of Pure Mathematics, Faculty of Mathematical Sciences,\\
  University of Kashan, Kashan 87317--53153, I. R. Iran\/} \\
 {\rm E-mail:} {\tt alighalavand@grad.kashanu.ac.ir,~ashrafi@kashanu.ac.ir}



 \end{center}

\begin{abstract}\noindent
Let $G$ be a finite simple graph with Laplacian polynomial $\psi(G,\lambda)$ $=$ $\sum_{k=0}^{n}$ $(-1)^{n-k}$ $c_k\lambda^k$. In an earlier paper,  the coefficients $c_{n-4}$ and $c_{n-5}$ for trees with respect to some degree-based graph invariants were computed. The aim of this paper is to continue this work by giving an  exact formula for the coefficient $c_{n-6}$. As a consequence of this work, the Laplacian coefficients $c_{n-k}$ of a forest $F$, $1 \leq k \leq 6$,  are computed in terms of the number of closed walks in $F$ and its line graph.

\vskip 3mm


\vskip 3mm

\end{abstract}

\section{\bf Definitions and Notations}
A \textit{simple undirected graph}  is a pair $G = (V,E)$ consisting of a set $V = V(G)$ of vertices and a set $E = E(G)$ of $2-$element subsets of $V$. The elements of $E$ are called \textit{edges} and the number of elements in $V$ is called the \textit{order} of $G$. The notations $n(G)$ and $m(G)$ denote the number of vertices and edges of $G$, respectively. There are two other graph notations worth mentioning now. The first one is $deg_G(v)$  which is the number of edges in $G$ with one end point $v$ and the second one is $deg_G(e)$ which is defined as the degree of vertex $e$ in the line graph of $G$. Obviously,  $deg_G(e) = deg_G(u) + deg_G(v) - 2$.

We use the notation $uvw$ to denote the path of length two such that vertices $u$ and $w$ have degree one and the vertex $v$ has degree two. In a similar way, we use the notation  $uvwx$ to denote a path of length three.

A graph $G$ is said to be \textit{connected} if for arbitrary vertices $x$ and $y$ in $V$, there exists a sequence $x=x_0, x_1, \ldots, x_r=y$ of vertices such that $x_ix_{i+1} \in E$,  $0 \leq i \leq r-1$.  The \textit{distance} between two vertices $u$ and $v$ in a connected graph $G$, $d_G(u,v)$, is defined as the length of a shortest path connecting these vertices and the sum of such numbers is called the \textit{Wiener index} of $G$, denoted by $W(G)$ \cite{16}. The hyper-Wiener index is a generalization of the Wiener index. It   was introduced for trees by Randi\'c in 1993 \cite{151} and for a general graph by  Klein et al. \cite{101}. This topological index is defined as $WW(G)$ $=$ $\frac{1}{2}\sum_{u,v\in V(G)}(d(u,v)+d^{2}(u,v))$.

A subgraph $H$ of a graph $G$ is a graph with vertex set $V(H)$ and edge set $E(H)$, such that $V(H)\subseteq\,V(G)$ and $E(H)\subseteq\,E(G)$.
We use the notation $H\leqslant\,G$ to denote that $H$ is subgraph of $G$. If $Z \subseteq V$, then the \textit{induced subgraph} $G[Z]$ is the graph with vertex set  $Z$ and  the edge set $\{uv\in E \mid \{u,v\}\subseteq Z\}$, and if $H\leqslant\,G$, then $G-H$ is a subgraph of $G$, with vertex set
$V(G)\backslash\,V(H)$ and edge set $E(G)\backslash\,\{uv\mid \{u,v\}\cap\,V(H)\neq\emptyset\}$.

The \textit{subdivision graph} $S$ is a graph constructed from $G$ by inserting a new vertex on each edge of $G$.  It is clear that $n(S) = n(G) + m(G)$ and $m(S) = 2m(G)$.

Suppose $G$ is a graph containing two edges $e$ and $f$. If the edges $e$ and $f$ have a common vertex $u$, then we write $u = e \cap f$. In the case that  $e$ and $f$  don't have common vertex, we will say the edges $e$ and $f$ are \textit{ independent}.  If $M \subseteq E(G)$ and all pair of edges in $M$ are independent, then the set $M$ is called a \textit{ matching} for $G$. A $k$-matching is a matching of size $k$, $1\leq k\leq \frac{n}{2}$, and the number of such matchings is denoted by $ m_k$.  The \textit{matching polynomial} of $G$, $\alpha$, is defined by $\alpha(G,x)$ $=$ $\sum_{k=0}^{\lfloor\frac{n}{2}\rfloor}(-1)^km_kx^{n-2k}$, where $x \ne 0$.  By definition $\alpha(G,0)=1$, see \cite{4} for more details.

In 1972, Gutman and Trinajsti\'c \cite{6} introduced the first degree-based graph invariant applicable in chemistry. This invariant is the first Zagreb index and can defined by the formula $M_1^2(G) = \sum_{v \in V}deg_G{(v)}^2$. The second Zagreb index $M_2^1(G) = \sum_{uv \in E}deg_G{(u)}deg_G{(v)}$ was introduced by Gutman et al. \cite{65} three years later in 1975. The complete history of these graph invariants together with the most important mathematical results about them are reported in \cite{7,8,14}.

The  \textit{forgotten index} of $G$ is another variant of the Zagreb group indices defined as  $M_1^3(G)$ $=$ $\sum_{v \in V}deg_G(v)^3$ $=$ $\sum_{e = uv \in E}[deg_G(u)^2 + deg_G(v)^2]$ \cite{5}. It can be see that $M_1^\alpha(G)$ $=$ $\sum_{u \in V}deg_G(u)^\alpha$, $\mathbb{R} \ni \alpha \ne 0, 1$ is  the general form of the first Zagreb index. Zhang and Zhang \cite{18} obtained the extremal values of the general Zagreb index  in the class of all unicyclic graphs. Mili\'cevi\'c et al. \cite{11}, reformulated the first and second Zagreb indices  in terms of the edge-degrees instead of the vertex-degrees. These  invariants were defined the first and second reformulated Zagreb indices defined as
$EM_1(G)$ $=$ $\sum_{e \cap f \ne \emptyset}[deg_G(e)+deg_G(f)]$ $=$ $\sum_{e\in E}deg_G(e)^2$ and $EM_2(G)$ $=$ $\sum_{e \cap f \ne \emptyset}deg_G(e)deg_G(f)$, respectively.

A $\{ 0,1\}$-matrix is a matrix whose entries consist only of the numbers $0$ and $1$. Suppose $G$ is a graph with vertex set $V = \{ u_1, \cdots, u_n\}$. The adjacency matrix of $G$ is a $\{ 0,1\}$-matrix $A (G)= (a_{ij})$  in which $a_{ij} = 1$ if and only if $u_iu_j \in E$. It is easy to see that $A$ is a real symmetric matrix of order $n$ and so all of its eigenvalues are real.  The matrices $D(G) = [d_{ij}]$ and $L(G) = D(G) - A(G)$ in which $d_{ii} = deg(u_i)$ and  $d_{ij} = 0, i \ne j,$ are called the diagonal and Laplacian matrices of $G$, respectively. It is well-known that all eigenvalues of $L(G)$ are non-negative real numbers with $0$ as the smallest eigenvalue.

The Laplacian polynomial of a graph $G$ is one of the most important polynomial associated to a graph. If $G$ is a graph, then the Laplacian polynomial of $G$  is the characteristic polynomial of  $L(G)$. The roots of this  polynomial are called the Laplacian eigenvalues of $G$. Suppose $\psi(G,x)$ $=$ $\det(x I_n - L)$ $=$ $\sum_{k=0}^{n}(-1)^{n-k}c_kx^k$ denotes the Laplacian polynomial of $G$. Since the coefficients of the Laplacian polynomial have graph theoretical meaning, some authors took into account the coefficients of this polynomial.

Let $f$ be a  topological index and $G$ be a graph. For simplifying our arguments, we usually write  $f$ as $f(G)$.

\begin{lemma} \label{l1}
Suppose $G$ is a graph. The following statements hold:
\begin{enumerate}
\item {\rm (Merris \cite{12} and Mohar \cite{13})} $c_0(G)=0$, $c_1(G)=n\tau(G)$, $c_n(G)=1$ and $c_{n-1}(G)=2m$, where $\tau(G)$ is the number of spanning trees of $G$;

\item  {\rm (Yan and Yeh \cite{17})} $c_2(G)=W(G)$, when $G$  is a tree;

\item  {\rm (Gutman \cite{8-8})} $c_3(G)=WW(G)$, when $G$  is a tree;

\item  {\rm (Oliveira et al. \cite{15})} $c_{n-2}(G)=\frac{1}{2}[4m^2-2m-M_1^2]$ and
$c_{n-3}(G)=\frac{1}{3!}[4m^2(2m-3)-6M_1^2m+6M_1^2+2M_1^3-12t(G)]$, where $t(G)$ is the number of triangles in $G$.
\end{enumerate}
\end{lemma}

In   \cite{1,2,2-2} we proved the following formulas for the coefficients $c_{n-4}(G)$ and $c_{n-5}(G)$, when $G$ is a forest, respectively.
\begingroup\allowdisplaybreaks\begin{align}
c_{n-4}(G) &=\frac{1}{4!}\Big[4m\big(4m^3-12m^2+51m-6M_1^2m-33M_1^2+4M_1^3+3\big)+3M_1^2\big(17M_1^2\nonumber\\
&-20\big)+72 M_1^3-54 M_1^4-24 M_2^1\Big]-16\hspace{-3mm}\sum_{\{u,v\}\subset V(G)}{deg_G(u)\choose 2}{deg_G(v)\choose 2}\nonumber\\
&=\frac{1}{4!}\Big[4m\big(4m^3-12m^2+3m-6M_1^2m+15M_1^2+4M_1^3+3\big)+3\big(M_1^2-2\big)^2\label{oeq1}\\
&-24M_1^3-6M_1^4-24M_2^1-12   \Big],\nonumber
\end{align}\endgroup
\begingroup\allowdisplaybreaks\begin{align}
c_{n-5}(G)&=\frac{1}{5!}\Big[2m\big(16m^4-80m^3+60m^2-40M_1^2m^2+60m+180M_1^2m
+40M_1^3m \label{oeq2}\\
&+15(M_1^2)^2-120M_1^2-140M_1^3-30M_1^4-120M_2^1\big)-20M_1^2\big(3M_1^2+M_1^3\nonumber\\
&+6\big)+120 M_1^3+120 M_1^4+24 M_1^5+240 M_2^1+120\alpha_{1,2}\Big].\nonumber
\end{align}\endgroup

Suppose $\lambda$ and $\xi$ are two arbitrary real numbers. We now define three invariants which is useful in simplifying formulas in our results. These are:
\begingroup\allowdisplaybreaks\begin{align*}
\alpha_{\lambda,\xi}(G) &= \sum\limits_{uv \in E}\Big[deg_G(u)^\lambda\,deg_G(v)^\xi+deg_G(u)^\xi\,deg_G(v)^\lambda\Big],\\
\beta(G) &=  \sum\limits_{e \sim f}\,deg_G(e \cap f)\Big(deg_G(e)+deg_G(f)\Big),\\
M_2^\lambda(G)&= \sum\limits_{uv\in E}\,\Big(deg_G(u)deg_G(v)\Big)^\lambda.
\end{align*}\endgroup
Note that the second Zagreb index is just the case of $\lambda = 1$ in $M_2^\lambda$.

Let $G$ and $H$ be  graphs. Set $\mathcal{S}_H(G)=\{X\mid X\leqslant\,G~\text{and}~X\cong\,H\}$. If $f$ and $g$ are two degree-based graph invariants, then we define two new degree-based topological indices $fg$ and $Hf$ as  $fg(G)$ $=$ $f(G)\times\,g(G)$ and
$Hf(G)$ $=$ $\sum_{X\in\mathcal{S}_H(G)}f(G-X)$.

Let $P_n$ denote the path graph on $n$ vertices. In a recent paper \cite{3}, Das et al.  presented the following formula for the number of  $k$-matchings,  $1\leq k\leq
\lfloor\frac{n}{2}\rfloor$, in a graph $G$ as:
\begin{equation}\label{5eq1}
m_k(G)=\frac{1}{k}P_2m_{k-1}(G).
\end{equation}

They also proved the following two results:

\begin{lemma}\label{5ml1}
Let $G$ be a graph with $n$ vertices and $m$ edges. Then
\begin{enumerate}
\item $P_2{m}(G)=m^2+m-M_1^2$.
\item $P_2{m^2}(G)=m^3+M_1^3-2mM_1^2+2M_2^1+2m^2-2M_1^2+m$.
\item $P_2{m^3}(G)=m^4-3m^2M_1^2+3mM_1^3+6mM_2^1-M_1^4-3\alpha_{1,2}+3m^3
-6mM_1^2+3M_1^3+6M_2^1+3m^2-3M_1^2+m$.
\item $P_2{m^4}(G)=M_1^5+4\alpha_{1,3}-4mM_1^4+6M_2^2-12m\alpha_{1,2}+6m^2M_1^3
+12m^2M_2^1-4m^3M_1^2-4M_1^4-12\alpha_{1,2}+12mM_1^3
+24mM_2^1-12m^2M_1^2+6M_1^3+12M_2^1-12mM_1^2-4M_1^2
+m^5+4m^4+6m^3+4m^2+m$,
\item $P_2{m^5}(G) = m^6 + 5m^5 + (10-5M_1^2)m^4 + (10+10M_1^3+20M_2^1-20M_1^2)m^3 + (5+60M_2^1-10M_1^4-30\alpha_{1,2} +30M_1^3-30M_1^2)m^2 + (30M_2^2-60\alpha_{1,2}+5M_1^5+20\alpha_{1,3}+30M_1^3-20M_1^2+60M_2^1-20M_1^4+1)m+20\alpha_{1,3}-5\alpha_{1,4} -10\alpha_{2,3}+10M_1^3-5M_1^2+20M_2^1-10M_1^4+5M_1^5-M_1^6+30M_2^2-30\alpha_{1,2}$.
\end{enumerate}
\end{lemma}

\begin{lemma}\label{5lm2}
Let $G$ be a graph with $n$ vertices, $m$ edges and girth $ \geq 5$. Then
\begin{enumerate}
\item $P_2{M_1^2}(G)=(m+3)\,M_1^2-M_1^3-4M_2^1-2m $.
\item $P_2{M_1^3}(G)=(m+3)\,M_1^3-M_1^4-3\alpha_{1,2}+6M_2^1-4M_1^2+2m$.
\item $P_2{M_1^4}(G) = (m+4)\,M_1^4-M_1^5+5M_1^2-2m -4\alpha_{1,3} + 6\alpha_{1,2}-6M_1^3-8M_2^1$.
\item $P_2{M_2^1}(G)=(m-9)M_2^1-2EM_2-5M_1^3+11M_1^2+M_1^4+\alpha_{1,2}-8m$.
\item $P_2{(mM_1^2)}(G)=(M_1^2-2)m^2+(4M_1^2-M_1^3-4M_2^1-18)m-6M_1^3+\alpha_{1,2}
-2\beta+17M_1^2-6M_2^1+8EM_1+4EM_2-(M_1)^2+M_1^4$.
\end{enumerate}
\end{lemma}

The following theorem is crucial in our main result \cite{2-2}.

\begin{theorem}\label{5th2}
Let $G$ be a graph with $m$ edges. Then
\begin{enumerate}
\item $m_5(S(G))=\frac{1}{15}m^2\Big[4m^3-20m^2+15m+15\Big]+\frac{1}{12}m\Big[8M_1^3m-8M_1^2m^2
 +3(M_1^2)^2+36M_1^2m-28M_1^3-24M_1^2-6M_1^4-24M_2^1\Big]+\alpha_{1,2}
 -\frac{1}{6}M_1^2\Big[3M_1^2+M_1^3+6\Big]+2M_2^1+\frac{1}{5}M_1^5+M_1^4+M_1^3.$
\item {\small $  M_1^2(S(G))=M_1^2+4m,~M_1^3(S(G))=M_1^3+8m,~M_1^4(S(G))=M_1^4+16m,
M_1^5(S(G))$} $=M_1^5+32m,~ \alpha_{1,2}(S(G))  =4M_1^2+2M_1^3,
\alpha_{1,3}(S(G))  =8M_1^2+2M_1^4,~\beta(S(G)) =2M_1^2+M_1^4-M_1^3,~
M_2^1(S(G))=2M_1^2,$ $~ M_2^2(S(G))=4M_1^3,~ EM_1(S(G))=M_1^3,~ EM_2(S(G)) =M_2^1+\frac{1}{2} M_1^4-\frac{1}{2}M_1^3. $
\item $P_2{(M_1^2)^2}(S(G))=(2m-10)(M_1^2)^2+(16m^2-2M_1^3-40m)M_1^2+32m^3-8mM_1^3
+13M_1^3+6M_1^4+M_1^5+24M_2^1+4\alpha_{1,2}$.
\item  $P_2{(m^2M_1^2)}(S(G))=32m^4+(8M_1^2-32)m^3-(4M_1^3+44M_1^2-8)m^2+(20M_1^3
-4(M_1^2)^2+30M_1^2+4M_1^4+16M_2^1)m+M_1^3M_1^2
+2(M_1^2)^2-7M_1^3-5M_1^2-5M_1^4-M_1^5 -8M_2^1-2\alpha_{1,2}$.
\item  $P_2{(mM_1^3)}(S(G))=32m^3+(4M_1^3-24)m^2-(8M_1^3+16M_1^2+2M_1^4)m+4m - (M_1^2-10)M_1^3+6M_1^2+M_1^4+M_1^5-6M_2^1+3\alpha_{1,2}$.
\item  $P_2{(mM_2^1)}(S(G))=\big(8M_1^2+8\big)m^2-(4M_1^3+10M_1^2+4M_2^1+4)m-2(M_1^2)^2
+2M_1^3+M_1^2+2M_1^4+8M_2^1+\alpha_{1,2}$.
\item $P_2{EM_2}(S(G)) = \frac{1}{2}m\big(4M_2^1-2M_1^3+2M_1^4+ 4\big)+\frac{11}{2}M_1^3-2\alpha_{1,2}-\frac{7}{2}M_1^2-\frac{3}{2}M_1^4-\frac{1}{2}M_1^5$.
\end{enumerate}

\end{theorem}

\section{Laplacian Coefficients and Degree-Based Invariants }
The aim of this section is to present an exact formula for the coefficient $c_{n-6}$ of the Laplacian polynomial in terms of some degree-based graphs invariants.

\begin{lemma}\label{6lmm0}
Let $G$ be a graph with $n$ vertices and $m$ edges. Then
$M_1^6(S(G))$ $=$ $M_1^6+64m$, $\alpha_{1,4}(S(G))$ $=$ $2M_1^5+16M_1^2$, $\alpha_{2,3}(S(G))$ $=$ $4M_1^4+8M_1^3$.
\end{lemma}

\begin{proof}
Apply definition of $S(G)$,  to prove that $M_1^6(S(G))$ $=$ $M_1^6$ $+$ $64m$, $\alpha_{1,4}(S(G))$ $=$ $\sum_{v\in V}\sum_{uv\in E}[2deg_G(v)^4+16deg_G(v)]$ $=$ $\sum_{v\in V}[2deg_G(v)^5+16deg_G(v)^2]$ $=$ $2M_1^5+16M_1^2$ and
$\alpha_{2,3}(S(G))$ $=$ $\sum_{v\in V}\sum_{uv\in E}[4deg_G(v)^3+8deg_G(v)^2]$
$=$ $\sum_{v\in V}[4deg_G(v)^4+8deg_G(v)^3]=4M_1^4+8M_1^3$.
\end{proof}

The next lemma is a direct consequence of Lemmas \ref{5ml1}, \ref{6lmm0} and Theorem \ref{5th2}(2).

\begin{lemma}\label{6lmm1}
Let $G$ be a graph with $n$ vertices and $m$ edges. Then
\begin{enumerate}
\item $P_2{m}(S(G))=4m^2-2m-M_1^2$.

\item $P_2{m^2}(S(G))=M_1^3+2M_1^2-2m(2M_1^2-4m^2+4m-1)$.

\item $P_2{m^3}(S(G))=16m^4-24m^3-(12M_1^2-12)m^2+(6M_1^3+12M_1^2-2)m-3M_1^3-3M_1^2-M_1^4$.

\item $P_2{m^4}(S(G)) = 32m^5-64m^4- (32M_1^2-48)m^3 + (24M_1^3+48M_1^2-16)m^2 - (24M_1^3+24M_1^2+8M_1^4-2)m+6M_1^3 +4M_1^2+4M_1^4+M_1^5$.

\item $P_2{m^5}(S(G))=64m^6-160m^5- (80M_1^2-160)m^4+ (80M_1^3+160M_1^2-80)m^3 - (120M_1^3+120M_1^2+40M_1^4-20)m^2 + (60M_1^3+40M_1^2+40M_1^4+10M_1^5-2)m-10M_1^3-5M_1^2-10M_1^4-5M_1^5-M_1^6$.
\end{enumerate}
\end{lemma}

It is easy to see that girth $S(G)$ $\geq6$. Therefore, Lemma \ref{5lm2} and Theorem \ref{5th2}(2) imply the following lemma:

\begin{lemma}\label{6lmm2}
Let $G$ be a graph on $n$ vertices and $m$ edges. Then
\begin{enumerate}
\item $P_2{M_1^2}(S(G))= (2m-5)M_1^2+8m^2-M_1^3$.

\item  $P_2{M_1^3}(S(G))=16m^2+(2M_1^3-4)m-3M_1^3-4M_1^2-M_1^4 $.

\item  $P_2{M_1^4}(S(G))=(2m-4)M_1^4+32m^2+6M_1^3-19M_1^2-M_1^5 $.

\item $P_2{M_2^1}(S(G))= 4mM_1^2-2M_1^3-3M_1^2-2M_2^1+4m$.

\item $P_2{(mM_1^2)}(S(G)) = 16m^3+ (4M_1^2-8)m^2-(2M_1^3+16M_1^2)m-(M_1^2)^2+4M_1^3+5M_1^2+4M_2^1+M_1^4$.
\end{enumerate}
\end{lemma}

\begin{lemma}\label{6lmm3}
Let $G$ be a graph with $n$ vertices and $m$ edges. Then
$P_2{M_1^5}(S(G))=2mM_1^5+64m^2-M_1^6-4m-5M_1^5+10M_1^4-10M_1^3-26M_1^2$.
\end{lemma}

\begin{proof}
Apply definition of $P_2{M_1^5}(S(G))$ to show that
$P_2{M_1^5}(S(G))$ $=$ $2mM_1^5(S(G))$ $-$ {\small$\sum_{v\in V}\sum_{uv}[deg_G(v)^5 +5deg_G(v)^4-10deg_G(v)^3$
$+$ $10deg_G(v)^2+26deg_G(v)+2]$ $=$ $2mM_1^5(S(G))$} $-$ $\sum_{v\in V}[deg_G(v)^6 +5deg_G(v)^5-10deg_G(v)^4$
$+$ $10deg_G(v)^3+26deg_G(v)^2+2deg_G(v)]$. Now the proof follows from Lemma \ref{5th2}(2) and simple calculations.
\end{proof}

\begin{lemma}\label{6lmm4}
Let $G$ be a graph with $n$ vertices and $m$ edges. Then
$P_2{(m^3M_1^2)}(S(G))=64m^5+(16M_1^2-96)m^4+(-8M_1^3-112M_1^2+48)m^3 + (-12(M_1^2)^2+72M_1^3+120M_1^2+12M_1^4+48M_2^1-8)m^2
+(12(M_1^2)^2+(6M_1^3-44)M_1^2-54M_1^3-48M_2^1-34M_1^4-6M_1^5-12\alpha_{1,2})m- 3(M_1^2)^2+(-3M_1^3-M_1^4+5)M_1^2+10M_1^3+12M_2^1+12M_1^4+6M_1^5+M_1^6+6\alpha_{1,2}+2\alpha_{1,3}$.
\end{lemma}

\begin{proof}
By definition of $S(G)$,
$P_2{(m^3M_1^2)}(S(G))$ $=$ $\sum_{v\in V}$ $\sum_{uv\in E}[2m-deg_G(v)-1]^3$ $[M_1^2$ $(S(G))-deg_G(v)^2-3deg_G(v)-2deg_G(u)]$.
Suppose
$X=16m^4M_1^2(S(G))+(-8M_1^3-24M_1^2-24M_1^2(S(G)))m^3+((-12M_1^2+12)M_1^2(S(G))+36M_1^2+12M_1^4)m^2
+((6M_1^3+12M_1^2-2)M_1^2(S(G))-42M_1^3-18M_1^2-6M_1^5-30M_1^4)m+ (-3M_1^3-3M_1^2-M_1^4)M_1^2(S(G))+10M_1^3+3M_1^2+12M_1^4+6M_1^5+M_1^6$. Then,  $P_2{(m^3M_1^2)}(S(G))$ $=$ $X$ $+$  {\small $\sum_{v\in V}\sum_{uv\in E}$ $[-16m^3deg_G(u)$ $+$ $2deg_G(u)$ $+$ $24m^2deg_G(u)deg_G(v)$ $+$ $48m^2deg_G(v)^2-12mdeg_G(u)deg_G(v)^2+2deg_G(u)deg_G(v)^3$} $+24mdeg_G(u)deg_G(v)+6deg_G(u)deg_G(v)^2-12mdeg_G(u)$
$+6deg_G(u)deg_G(v)]$. We now replace $\sum_{v\in V}\sum_{uv\in E}$ by $\sum_{uv\in E}$ to show that
$P_2{(m^3M_1^2)}(S(G))$ $=$ $X$ $-$ $16m^3M_1^2$
$+(48M_1^3+24M_1^2+48M_2^1)m^2+(-12M_1^2-48M_2^1-12\alpha_{1,2})m+2M_1^2+12M_2^1+6\alpha_{1,2}+2\alpha_{1,3}$. The proof now follows from Lemma \ref{5th2}(2).
\end{proof}

\begin{lemma}\label{6lmm5}
Let $G$ be a graph with $n$ vertices and $m$ edges. Then
$P_2{(m^2M_1^3)}(S(G))=64m^4+(8M_1^3-80)m^3+(-20M_1^3-48M_1^2-4M_1^4+32)m^2+ \big((-4M_1^2+50)M_1^3+40M_1^2-24M_2^1+4M_1^4
+4M_1^5+12\alpha_{1,2}-4)m+(M_1^3)^2+(2M_1^2-19)M_1^3-8M_1^2+12M_2^1-8M_1^4- 2M_1^5-M_1^6-6M_2^2-3\alpha_{1,2}$.
\end{lemma}

\begin{proof}
Apply definition of $S(G)$ to show that
$P_2{(m^2M_1^3)}(S(G))$ $=$ $\sum_{v\in V}\sum_{uv\in E}[2m-deg_G(v)-1]^2[M_1^3(S(G))-deg_G(v)^3-3deg_G(u)^2
+3deg_G(u)-7deg_G(v)-2]$. Now  a similar argument as  Lemma \ref{6lmm4} completes the proof.
\end{proof}

\begin{lemma}\label{6lmm6}
Let $G$ be a graph with $n$ vertices and $m$ edges. Then
$P_2{(mM_1^4)}(S(G))=(4M_1^4+64m)m^2-32m^2+(12M_1^3-54M_1^2-10M_1^4-2M_1^5)m +(4-M_1^2)M_1^4+9M_1^3+19M_1^2+8M_2^1+M_1^5+M_1^6-6\alpha_{1,2}+4\alpha_{1,3}$.
\end{lemma}

\begin{proof}
Apply definitions of $S(G)$ and $P_2{(mM_1^4)}$ to write the   form $P_2{(mM_1^4)}(S(G))$ $=$ $\sum_{v\in V}\sum_{uv\in E}[2m-deg_G(v)-1][M_1^4(S(G))-deg_G(v)^4
-4deg_G(u)^3 +6deg_G(u)^2-4deg_G(u)-15deg_G(v)]$. Now a similar argument as  Lemma \ref{6lmm4} gives the proof.
\end{proof}
\begin{lemma}\label{6lmm7}
Let $G$ be a graph with $n$ vertices and $m$ edges. Then
$P_2{\alpha_{1,2}}(S(G))=2\big(4M_1^2+2M_1^3\big)m-9M_1^3+M_1^2-2M_1^4-6M_2^1-\alpha_{1,2}+4m$.
\end{lemma}

\begin{proof}
By definitions of $S(G)$ and $\alpha_{1,2}$, we can write
{\small\begingroup\allowdisplaybreaks\begin{align*}
P_2{\alpha_{1,2}}(S(G))&=2m\alpha_{1,2}(S(G))-\sum_{v\in V}\Big[ \big(2deg_G(v)^2+4deg_G(v)\big)\big(deg_G(v)^2+deg_G(v)\big)\\
&+\big(2deg_G(v)^2+4deg_G(v)\big)\big(deg_G(v)^2-deg_G(v)\big)\\
&-\big(2(deg_G(v)-1)^2+4(deg_G(v)-1)\big)\big(deg_G(v)^2-deg_G(v)\big)\Big]\\
&-\sum_{v\in V}\sum_{uv\in E} \Big(\big(2deg_G(v)^2+4deg_G(v)\big)-\big(deg_G(v)^2+deg_G(v)\big)\Big)\big(deg_G(u)-1\big).
\end{align*}\endgroup}

Now by simple calculations,  we obtain
\begingroup\allowdisplaybreaks\begin{align*}
P_2{\alpha_{1,2}}(S(G))&=2m\alpha_{1,2}(S(G))-10M_1^3-2M_1^2-2M_1^4+4m\\
&-\sum_{uv\in E}\Big(deg_G(u)deg_G(v)^2+deg_G(u)^2deg_G(v)+6deg_G(u)deg_G(v)\\
&-deg_G(u)^2-deg_G(v)^2-3deg_G(u)-3deg_G(v)\Big)\\
&=2m\alpha_{1,2}(S(G))-10M_1^3-2M_1^2-2M_1^4+4m-\alpha_{1,2}\\
&-6M_2^1+M_1^3+3M_1^2,
\end{align*}\endgroup
and Lemma \ref{5th2}(2) gives the result.
\end{proof}


\begin{lemma}\label{6lmm8}
Let $G$ be a graph with $n$ vertices and $m$ edges. Then
$P_2{(m(M_1^2)^2)}(S(G))=64m^4+(32M_1^2-32)m^3+(4(M_1^2)^2-16M_1^3-112M_1^2)m^2 +(-30(M_1^2)^2+(-4M_1^3+40)M_1^2+58M_1^3
+80M_2^1+20M_1^4+2M_1^5+8\alpha_{1,2})m-(M_1^2)^3+10(M_1^2)^2+(8M_1^3+8M_2^1+ 2M_1^4)M_1^2-13M_1^3-24M_2^1-15M_1^4-7M_1^5-M_1^6-20\alpha_{1,2}-4\alpha_{1,3}$.
\end{lemma}

\begin{proof}
By definition of {\small$S(G)$, $P_2{(m(M_1^2)^2)}(S(G))$ $=$ $\sum_{v\in V}\sum_{uv\in E}[2m-deg_G(v)-1][M_1^2(S(G))$} $-deg_G(v)^2-3deg_G(v)-2deg_G(u)]^2$.
 Now the proof follows from a similar argument as Lemma \ref{6lmm4}.
\end{proof}

\begin{lemma}\label{6lmm9}
Let $G$ be a graph with $n$ vertices and $m$ edges. Then
$P_2{(M_1^2M_1^3)}(S(G))=64m^3+(8M_1^3+16M_1^2-16)m^2+((2M_1^3-60)M_1^2-20M_1^3-4M_1^4)m-4(M_1^2)^2+(-8M_1^3
-M_1^4+10)M_1^2-(M_1^3)^2+17M_1^3+2\alpha_{1,3}+10M_2^1+13M_1^4+3M_1^5+M_1^6+6M_2^2+6\alpha_{1,2}$.
\end{lemma}

\begin{proof}
By definition of $S(G)$, $P_2{(M_1^2M_1^3)}(S(G))$ $=$ $\sum_{v\in V}\sum_{uv\in E}[M_1^2(S(G))-deg_G(v)^2-3deg_G(v)-2deg_G(u)][M_1^3(S(G))$ $-deg_G(v)^3-3deg_G(u)^2+3deg_G(u)-7deg_G(v)-2]$. Now the proof can be completed in a similar way as Lemma \ref{6lmm4}.
\end{proof}

Define nine graph invariants as

\hspace{2cm}$\Theta_1(G)$ $=$ $\sum_{uvw\in  \mathcal{S}_{P_3}(G)}$ $[deg_G(u)deg_G(v)deg_G(w)],$\

\hspace{2cm}$\Theta_2(G)$ $=$ $\sum_{uvw\in \mathcal{S}_{P_3}(G)}$ $[deg_G(u)deg_G(w)],$\

\hspace{2cm}$\Theta_3(G)$ $=$ $\sum_{uvwx\in  \mathcal{S}_{P_4}(G)}$ $[deg_G(u)deg_G(x)],$\

\hspace{2cm}$\Theta_4(G)$ $=$ $\sum_{uvw\in  \mathcal{S}_{P_3}(G)}$ $[deg_G(u)^2deg_G(w)$ $+$ $deg_G(u)deg_G(w)^2],$\

\hspace{2cm}$\Theta_5(G)$ $=$ $\sum_{uvw\in  \mathcal{S}_{P_3}(G)}$ $deg_G(v)^2[deg_G(u)$ $+$ $deg_G(w)]$,\

\hspace{2cm}$\Theta_6(G)$ $=$ $\sum_{uvwx\in  \mathcal{S}_{P_4}(G)}$ $[deg_G(u)deg_G(v)$ $+$ $deg_G(w)deg_G(x)],$\

\hspace{2cm}$\Theta_2^+(G)$ $=$ $\sum_{uvw\in  \mathcal{S}_{P_3}(G)}$ $[deg_G(u)$ $+$  $deg_G(w)],$\

\hspace{2cm}$\Theta_3^+(G)$ $=$ $\sum_{uvwx\in  \mathcal{S}_{P_4}(G)}$ $[deg_G(u)$ $+$ $deg_G(x)]$,\

\hspace{1.8cm}$\Theta_3^{+,2}(G)$ $=$ $\sum_{uvwx\in  \mathcal{S}_{P_4}(G)}$ $[deg_G(u)^2$ $+$ $deg_G(x)^2]$.

\begin{lemma}\label{6lmm10}
Let $G$ be a graph with $n$ vertices and $m$ edges. Then $P_2{(M_1^2)^2}(G)$ $=$ $(m+6)(M_1^2)^2+(-2M_1^3-4m-8M_2^1-12)M_1^2+M_1^3$ $+$ $M_1^5+2M_2^2-6\alpha_{1,2}+4\alpha_{1,3}+4m+26M_2^1-2M_1^4$ $+8\Theta_1-24\Theta_2+8\Theta_3+4\Theta_4$.
\end{lemma}

\begin{proof}
By definition of $M_1^2$ and some tedious calculations, one can see that
\begingroup\allowdisplaybreaks\begin{align*}
P_2{((M_1^2)^2)}(G)& = \sum_{uv\in E}\Big[M_1^2-deg_G(u)^2-deg_G(v)^2-\sum_{xu\in E(G-uv)}[2deg_G(x)-1]\\ &-\sum_{yv\in E(G-uv)}[2deg_G(y)-1]\Big]^2\\
&=(m+2)(M_1^2)^2+(-2M_1^3-4m-4)M_1^2+5M_1^3-2M_1^4\\
&+M_1^5+2M_2^1+2M_2^2-2\alpha_{1,2}+4m\\
&+\sum_{uv\in E}\Big[4deg_G(u)^2\sum_{xu\in E(G-uv)}deg_G(x)+4deg_G(u)^2\sum_{yv\in E(G-uv)}deg_G(y)\\
&+4deg_G(v)^2\sum_{xu\in E(G-uv)}deg_G(x)+4deg_G(v)^2\sum_{yv\in E(G-uv)}deg_G(y)\\
&-(4M_1^2-8)\sum_{xu\in E(G-uv)}deg_G(x)-(4M_1^2-8)\sum_{yv\in E(G-uv)}deg_G(y)\\
&-4deg_G(u)\sum_{xu\in E(G-uv)}deg_G(x)-4deg_G(u)\sum_{yv\in E(G-uv)}deg_G(y)\\
&-4deg_G(v)\sum_{xu\in E(G-uv)}deg_G(x)-4deg_G(v)\sum_{yv\in E(G-uv)}deg_G(y)\\
&+\Big(\sum_{xu\in E(G-uv)}2deg_G(x)\Big)^2+\Big(\sum_{yv\in E(G-uv)}2deg_G(y)\Big)^2\\
&+2\Big(\sum_{xu\in E(G-uv)}2deg_G(x)\Big)\Big(\sum_{vy\in E(G-uv)}2deg_G(y)\Big)\Big].
\end{align*}\endgroup

Therefore, $P_2{(M_1^2)^2}(G)$ $=$ $(m+6)(M_1^2)^2+(-2M_1^3-4m-8M_2^1-12)M_1^2+M_1^3$
$+$ $M_1^5+2M_2^2-6\alpha_{1,2}+4\alpha_{1,3}+4m+26M_2^1-2M_1^4$
$+$ $8\Theta_1-24\Theta_2+8\Theta_3+4\Theta_4$, proving the lemma.
\end{proof}


\begin{lemma}\label{6lmm11}
Let $G$ be a graph with $n$ vertices and $m$ edges. Then
$P_2{(mM_1^3)}(G)=(M_1^3+2)m^2+(4M_1^3-3\alpha_{1,2}-4M_1^2+6M_2^1-M_1^4+2)m+(-M_1^2+4)M_1^3-9\alpha_{1,2}+\alpha_{1,3}-6M_1^2+14M_2^1-6\Theta_2+3\Theta_4-M_1^4+M_1^5+6M_2^2$.
\end{lemma}

\begin{proof}
Choose $uv \in E$ and set $A(uv)$ $=$ $\sum_{xu\in E(G-uv)}[3deg_G(x)^2-3deg_G(x)+1]$ and $B(uv)$ $=$ $\sum_{yv\in E(G-uv)}[3deg_G(y)^2-3deg_G(y)+1]$. Apply definition of $P_2{(mM_1^3)}$ to prove that $P_2{(mM_1^3)}(G)$ $=$ $\sum_{uv\in E}[m-deg_G(u)-deg_G(v)+1][M_1^3-deg_G(u)^3-deg_G(v)^3-A(uv)-B(uv)]$. Therefore
\begingroup\allowdisplaybreaks\begin{align*}
P_2{(mM_1^3)}(G)&=m^2M_1^3-M_1^2M_1^3+mM_1^3-mM_1^2-mM_1^4+2m^2+M_1^3-3M_1^2-M_1^4\\
&+M_1^5+2M_2^1+\alpha_{1,3}+2m
-3\sum_{uv\in E}\Big[\Big( m-deg_G(u)-deg_G(v)+1   \Big)\Big(\\
&\sum_{xu\in E(G-uv)}[deg_G(x)^2-deg_G(x)]+ \sum_{yv\in E(G-uv)}[deg_G(y)^2-deg_G(y)]\Big)  \Big].
\end{align*}\endgroup
Now, a similar argument as  Lemma \ref{6lmm10} gives our result.
\end{proof}


\begin{lemma}\label{6lmm12}
Let $G$ be a graph with $n$ vertices and $m$ edges. Then
$P_2{(m^2M_1^2)}(G)=(M_1^2-2)m^3+(-M_1^3+5M_1^2-4M_2^1-4)m^2+(-2(M_1^2)^2-4M_1^3+11M_1^2+2M_1^4-20M_2^1+6\alpha_{1,2}+8\Theta_2-2)m-2(M_1^2)^2
+(M_1^3+2M_2^1+7)M_1^2-5M_1^3+11\alpha_{1,2}-4\alpha_{1,3}-2\Theta_4-$ $20M_2^1-8\Theta_1+8\Theta_2
+3M_1^4-M_1^5-2M_2^2$.
\end{lemma}

\begin{proof}
Apply definition of $P_2{(m^2M_1^2)}$ and some tedious calculations to show that
\begingroup\allowdisplaybreaks\begin{align*}
&P_2{(m^2M_1^2)}(G)=(M_1^2-2)m^3+(-M_1^3+3M_1^2-4)m^2+(-2(M_1^2)^2-4M_1^3\\
&+7M_1^2+2M_1^4
-4M_2^1+2\alpha_{1,2}-2)m-2(M_1^2)^2+(M_1^3+2M_2^1+5)M_1^2\\
&-5M_1^3-8M_2^1+3M_1^4
-M_1^5-2M_2^2+5\alpha_{1,2}-2\alpha_{1,3}\\
&-2\sum_{uv\in E}\Big[m-deg_G(u)-deg_G(v)+1  \Big]^2\Big[\sum_{xu\in E(G-uv)}deg_G(x)+\sum_{yv\in E(G-uv)}deg_G(y)\Big].
\end{align*}\endgroup
Now a similar argument as  Lemma \ref{6lmm10},  gives the proof.
\end{proof}

\begin{lemma}\label{6lmm13}
Let $G$ be a graph with $n$ vertices and $m$ edges. Then
$EM_2(G)=\alpha_{1,2}-6M_2^1+\frac{1}{2}M_1^4-\frac{5}{2}M_1^3+6M_1^2-4m+\Theta_2$.
\end{lemma}

\begin{proof}
By definition of $EM_2$, $EM_2(G)$ $=$ $\sum_{uvw\in \mathcal{S}_{P_3}(G)}[deg_G(u)+deg_G(v)-2][deg_G(v)+deg_G(w)-2]$
$=$ $\alpha_{1,2}$ $-$ $6M_2^1$ $+$ $\frac{1}{2}M_1^4$ $-$ $\frac{5}{2}M_1^3$ $+$ $6M_1^2$ $-$ $4m$ $+$ $\Theta_2$.
\end{proof}


\begin{lemma}\label{6lmm14}
Let $G$ be a graph with $n$ vertices and $m$ edges. Then
$P_2{\Theta_2}(G)=(m+2)\Theta_2-\Theta_1-\Theta_4-2\Theta_3+\Theta_3^+$.
\end{lemma}

\begin{proof}
By definition of $P_2{\Theta_2}$,
\begingroup\allowdisplaybreaks\begin{align*}
P_2{\Theta_2}(G)&=m\Theta_2-\sum_{u v w\in\mathcal{S}_{P_3}(G)}deg_G(u)deg_G(w)\big( deg_G(u)+deg_G(v)+deg_G( w)-2 \big)\\
&-\sum_{x yab\in \mathcal{S}_{P_4}(G)}\big( 2deg_G(x)deg_G(b)-deg_G(x)-deg_G( b) \big)\\
&=(m+2)\Theta_2-\Theta_1-\Theta_4-2\Theta_3+\Theta_3^+,
\end{align*}\endgroup
as desired.
\end{proof}

By a similar arguments as  Lemma \ref{6lmm10}, we can prove the following two lemmas:

\begin{lemma}\label{6lmm15}
Let $G$ be a graph with $n$ vertices and $m$ edges. Then
$P_2{(mM_2^1)}(G)=m^2M_2^1+(2M_2^1-2\Theta_2-\alpha_{1,2}+\Theta_2^+)m-(M_1^2-1)M_2^1
+\Theta_2^++2\Theta_1-4\Theta_2+2\Theta_3+\Theta_4+\Theta_5$.
\end{lemma}


\begin{lemma}\label{6lmm16}
Let $G$ be a graph with $n$ vertices and $m$ edges. Then
$P_2{\alpha_{1,2}}(G)=(m+1)\alpha_{1,2}-2M_2^2-\alpha_{1,3}+\Theta_3^+-\Theta_3^{+,2}-2\Theta_6$.
\end{lemma}

For the sake of completeness, we mention here two results which are useful in our next calculations.
\begin{lemma}\label{6lmm17}
\cite{1}
Let $G$ be a graph. Then
$\beta(G)=\alpha+M_1^4-3M_1^3+2M_1^2-2M_2^1$.
\end{lemma}

\begin{theorem}\label{dasth1}
\cite{4}
Let $G$ be a graph with $m$ edges and girth $\geq 5$. Then
 $m_5(G)=\frac{1}{5!}\Big[m(m^4+10m^3+43m^2+54m-328)+30(M_1^2)^2-12\alpha_{1,2}(m-7)
 -20\alpha_{1,3}
-2M_1^2(2m^3+30m^2+61m-225)+12\beta+2M_2^1(6m^2+66m-239)
+M_1^3(6m^2+24m-149)+2M_1^4(m+10)+6M_2^2-EM_2-5M_1^5
+3P_2(M_1^2)^2+8P_2(mM_1^3)
-6P_2(m^2M_1^2)-P_2EM_2
+P_2(mM_2^1)\Big]$.
\end{theorem}

Now, Lemmas \ref{5ml1}(1), \ref{5lm2}, \ref{6lmm10}, \ref{6lmm11}, \ref{6lmm12}, \ref{6lmm13}, \ref{6lmm14}, \ref{6lmm15}, \ref{6lmm16}, \ref{6lmm17} and Theorem \ref{dasth1} give the following theorem.

\begin{theorem}\label{the1}
Let $G$ be a graph with $m$ edges and girth $\geq 5$. Then
$m_5(G)=\frac{1}{5!}\Big[m^5+10m^4-(10M_1^2-55)m^3+\big(60M_2^1-90M_1^2
+20M_1^3
+190\big)m^2+\big(15(M_1^2)^2+140M_1^3-376M_1^2-30M_1^4
+492M_2^1-120\alpha_{1,2}
-120\Theta_2+24\Theta_2^++336\big)m
+60(M_1^2)^2
-(60M_2^1+20M_1^3+768)M_1^2-120M_1^4
+24M_1^5+120M_2^2-504\alpha_{1,2}
+96\alpha_{1,3}+24\Theta_2^+-48\Theta_3^++96\Theta_4+24\Theta_3^{+,2}
+336M_1^3
+1440M_2^1+132\Theta_1-600\Theta_2+120\Theta_3+24\Theta_5+48\Theta_6\Big]$.
\end{theorem}


\begin{lemma}\label{6lmm18}
Let $G$ be a graph with $n$ vertices and $m$ edges. Then
${\Theta_1}(S(G))=2[M_2^1+M_1^3-M_1^2]$,
${\Theta_2}(S(G))= M_2^1+2M_1^2-4m$,
 ${\Theta_2^+}(S(G))=3M_1^2-4m$,
 ${\Theta_3}(S(G))=4M_2^1-2M_1^2$,
 ${\Theta_3^+}(S(G))=2M_2^1+M_1^2-4m$,
${\Theta_3^{+,2}}(S(G))=\alpha_{1,2}+4M_1^2-M_1^3-8m$,
 ${\Theta_4}(S(G))=\alpha_{1,2}+8M_1^2-16m$,
 ${\Theta_5}(S(G))=4M_1^2+2M_1^4-2M_1^3$, and
 ${\Theta_6}(S(G))=4M_2^1+2M_1^3-4M_1^2$.
\end{lemma}

\begin{proof}
The proof is straightforward and so it is omitted.
\end{proof}

\begin{lemma}\label{6lmm19}
Let $G$ be a graph with $n$ vertices and $m$ edges. Then
$P_2{\Theta_2}(S(G))=(4M_1^2+2M_2^1+4)m -8m^2-2M_1^3+3M_1^2-6M_2^1-\alpha_{1,2}$.
\end{lemma}
\begin{proof}
Lemmas \ref{6lmm14}, \ref{6lmm18} and \ref{6lmm1}(1) give our result.
\end{proof}



\begin{lemma}\label{6lmm20}
Let $G$ be a graph with $n$ vertices and $m$ edges. Then
\begin{enumerate}
\item $P_2{\Theta_1}(S(G))=(4M_1^3-4M_1^2+4M_2^1-8)m-2M_1^3+10M_1^2-2M_1^4-4\alpha_{1,2}$,
\item $P_2{\Theta_2^+}(S(G))=(6M_1^2+8)m-8m^2-3M_1^3-M_1^2-4M_2^1$,
\item $P_2{\Theta_3}(S(G))=(8M_2^1-4M_1^2-4)m+2M_1^3+6M_1^2-4M_2^1-4\alpha_{1,2}-2\Theta_2+\Theta_2^+  $,
\item $P_2{\Theta_3^+}(S(G))=2m(2M_2^1+M_1^2-4m)-M_1^3+6M_1^2-6M_2^1-2\alpha_{1,2}-\Theta_2^+  $,
\item $P_2{\Theta_3^{+,2}}(S(G))=(8M_1^2-2M_1^3+2\alpha_{1,2}+4)m-16m^2-M_1^3+3M_1^2-2M_2^1+M_1^4-2M_2^2-2\alpha_{1,2}-\alpha_{1,3}-3\Theta_2^+  $,
\item $P_2{\Theta_4}(S(G))=(16M_1^2+2\alpha_{1,2}+12)m-32m^2-7M_1^3+11M_1^2-14M_2^1-2M_2^2-3\alpha_{1,2}-\alpha_{1,3} $,
\item $P_2{\Theta_5}(S(G))=(4M_1^4-4M_1^3+8M_1^2+16)m+13M_1^3-18M_1^2-8M_2^1-5M_1^4-2M_1^5+\alpha_{1,2}-\alpha_{1,3} $,
\item $P_2{\Theta_6}(S(G))=(4M_1^3-8M_1^2+8M_2^1-12)m+M_1^3+13M_1^2+2M_2^1-2M_1^4-7\alpha_{1,2}$,
\item $P_2{(m\Theta_2)}(S(G))=  (8M_1^2+4M_2^1+16)m^2-16m^3+(4M_1^2-4M_1^3-10M_2^1-2\alpha_{1,2}-2\Theta_2^+-4)m-2(M_1^2)^2
-(M_2^1+4)M_1^2-2M_1^3+8M_2^1+2\Theta_2+2M_1^4+3\alpha_{1,2}+\alpha_{1,3}+\Theta_2^+$,
\item $P_2{(m\Theta_2^+)}(S(G))=(12M_1^2+24)m^2 -16m^3-(6M_1^3+4M_1^2+8M_2^1+8)m-3(M_1^2)^2-3M_1^2+12M_2^1+3M_1^4+2\alpha_{1,2}$.
\end{enumerate}

\end{lemma}

\begin{proof}
The proof has tedious calculations similar to  Lemma \ref{6lmm10}.
\end{proof}

\begin{lemma}\label{6lmm21}
Let $G$ be a graph with $n$ vertices and $m$ edges. Then
\begin{enumerate}
\item $P_2{\alpha_{1,3}}(S(G))=(16M_1^2+4M_1^4+20)m+4M_1^3-17M_1^2-14M_2^1-7M_1^4-2M_1^5-\alpha_{1,3} $.
\item $P_2{(m\alpha_{1,2})}(S(G))=(8M_1^3+16M_1^2+8)m^2-(22M_1^3+6M_1^2+12M_2^1+4M_1^4+2\alpha_{1,2}+4)m
-4(M_1^2)^2-(2M_1^3+3)M_1^2+9M_1^3+4M_2^1+6M_1^4+2M_1^5+2M_2^2+9\alpha_{1,2} $.
\item $P_2{(M_1^2M_2^1)}(S(G))=(4m-13)(M_1^2)^2+(16m^2-4M_1^3-2M_2^1-8m-10)M_1^2+16m^2-(8M_1^3+8M_2^1)m+6M_1^3
+18M_2^1+4\Theta_2+6M_1^4+2M_1^5+10\alpha_{1,2}+\alpha_{1,3}$.
\item $P_2{(m^2M_2^1)}(S(G))=(16M_1^2+16)m^3-(8M_1^3+28M_1^2+8M_2^1+16)m^2+(8M_1^3-8(M_1^2)^2+8M_1^2
+8M_1^4+32M_2^1+4\alpha_{1,2}+4)m
+4(M_1^2)^2+(2M_1^3+1)M_1^2-14M_2^1-4M_1^4-2M_1^5-5\alpha_{1,2}-\alpha_{1,3}   $.
\item $P_2{M_2^2}(S(G))=(8m-9)M_1^3-8m+12M_1^2-4M_1^4-3\alpha_{1,2} $.
\end{enumerate}
\end{lemma}

\begin{proof}
By definitions of $S(G)$ we have
{\small
\begingroup\allowdisplaybreaks\begin{align*}
P_2{\alpha_{1,3}}(S(G))&=2m{\alpha_{1,3}}(S(G))-\sum_{v\in V}\sum_{uv\in E}\Big[\big(2deg_G(v)^3+8deg_G(v)\big)deg_G(v) +
2deg_G(u)^3\\
&+8deg_G(u)+\big(2deg_G(u)^3+8deg_G(u)  \big)\big(deg_G(u)-1  \big)-\Big(2(deg_G(u)-1)^3\\
&+8(deg_G(u)-1)  \Big)\big(deg_G(u)-1  \big)+\sum_{xv\in E(G-uv)}\big(deg_G(x)^3+7deg_G(x)\big) \Big]\\
&=(16M_1^2+4M_1^4+20)m+4M_1^3-17M_1^2-14M_2^1-7M_1^4-2M_1^5-\alpha_{1,3},\\
P_2{(m\alpha_{1,2})}(S(G))&=\sum_{v\in V}\sum_{uv\in E}\Big[2m-deg_G(v)-1\Big]\Big[{\alpha_{1,2}}(S(G))-\big(2deg_G(v)^2+4deg_G(v)\\
&\big)deg_G(v) -2deg_G(u)^2-4deg_G(u)-\big(2deg_G(u)^2+4deg_G(u)  \big)\big(deg_G(u)-1  \big)\\
&+\Big(2(deg_G(u)-1)^2-4(deg_G(u)-1)  \Big)\big(deg_G(u)-1  \big)\\
&-\sum_{xv\in E(G-uv)}\big(deg_G(x)^2+3deg_G(x)\big) \Big]\\
&=(8M_1^3+16M_1^2+8)m^2-(22M_1^3+6M_1^2+12M_2^1
+4M_1^4+2\alpha_{1,2}+4)m\\
&-4(M_1^2)^2-(2M_1^3+3)M_1^2+9M_1^3+4M_2^1+6M_1^4+2M_1^5+2M_2^2+9\alpha_{1,2},\\
P_2{(M_1^2M_2^1)}(S(G))&=\sum_{v\in V}\sum_{uv\in E}\Big[M_1^2(S(G))-deg_G(v)^2-3deg_G(v)-2deg_G(u)\Big]\Big[M_2^1(S(G))\\
&-2deg_G(v)^2 -2deg_G(u)-2deg_G(u)\big(deg_G(u)-1  \big)+2\big(deg_G(u)-1\big)^2\\
&-\sum_{xv\in E(G-uv)}deg_G(x) \Big]\\
&=(4m-13)(M_1^2)^2+(16m^2-4M_1^3-2M_2^1-8m-10)M_1^2+16m^2\\
&-(8M_1^3+8M_2^1)m+6M_1^3+18M_2^1+4\Theta_2
+6M_1^4+2M_1^5+10\alpha_{1,2}+\alpha_{1,3},\\
P_2{(m^2M_2^1)}(S(G))&=\sum_{v\in V}\sum_{uv\in E}\Big[2m-deg_G(v)-1\Big]^2\Big[M_2^1(S(G))
-2deg_G(v)^2 -2deg_G(u)\\
&-2deg_G(u)\big(deg_G(u)-1  \big)
+2\big(deg_G(u)-1\big)^2-\sum_{xv\in E(G-uv)}deg_G(x) \Big]\\
&= (16M_1^2+16)m^3-(8M_1^3+28M_1^2+8M_2^1+16)m^2+(8M_1^3-8(M_1^2)^2\\
&+8M_1^2+8M_1^4+32M_2^1+4\alpha_{1,2}+4)m
+4(M_1^2)^2+(2M_1^3+1)M_1^2-14M_2^1\\
&-4M_1^4-2M_1^5-5\alpha_{1,2}-\alpha_{1,3},\\
P_2{M_2^2}(S(G))&=\sum_{v\in V}\sum_{uv\in E}\Big[M_2^1(S(G))-4deg_G(v)^3 -4deg_G(u)^2-4deg_G(u)^2\big(deg_G(u)\\
&-1  \big)+4\big(deg_G(u)-1\big)^3-\sum_{xv\in E(G-uv)}3deg_G(x)^2 \Big]\\
&=(8m-9)M_1^3-8m+12M_1^2-4M_1^4-3\alpha_{1,2},
\end{align*}\endgroup}
proving the lemma.
\end{proof}

\begin{theorem}\label{66th1}
Let $G$ be a graph with $m$ edges. Then
$m_6(S(G))=\frac{1}{6!}\Big[64m^6-480m^5+720m^4+600m^3-360m^2-480m
-720M_1^5-2160\alpha_{1,2}-720\alpha_{1,3}+540(M_1^2)^2-2340m^2M_1^2+2160mM_1^3-1080M_1^4+720M_2^1-240M_1^2M_1^3m-
120M_1^6-720M_2^2+600M_1^2M_1^3+1680M_1^2m^3-810(M_1^2)^2m-1920M_1^3m^2+1620M_1^4m+3600M_2^1m-720\Theta_2
-1260mM_1^2+720M_1^2+480M_1^3-240m^4M_1^2+320m^3M_1^3+360M_1^2M_2^1
+90M_1^2M_1^4+180(M_1^2)^2m^2-360M_1^4m^2-1440M_2^1m^2+288M_1^5m+1440\alpha_{1,2}m+40(M_1^3)^2-15(M_1^2)^3\Big]$.
\end{theorem}
\begin{proof}
The proof follows from Theorems \ref{the1} and  \ref{5th2},  Equation \ref{5eq1} and Lemmas \ref{6lmm1},
\ref{6lmm2}, \ref{6lmm3}, \ref{6lmm4}, \ref{6lmm5}, \ref{6lmm6}, \ref{6lmm7}, \ref{6lmm8}, \ref{6lmm9}, \ref{6lmm19}, \ref{6lmm20} and \ref{6lmm21}.
\end{proof}

We mention here a useful result of Zhou and Gutman \cite{19}.

\begin{theorem}\label{66th2}
Let $G$ be an $n-$vertex  forest. Then $c_{n-k}(G)=m_k(S(G))$, for $0\leq k\leq n$.
\end{theorem}

We are now ready to prove the main result of this section.

\begin{theorem}\label{66th3}
Let $G$ be a  forest with $m$ edges. Then
$c_{n-6}(G)=\frac{1}{6!}\Big[64m^6-480m^5+720m^4+600m^3-360m^2-480m
-720M_1^5-2160\alpha_{1,2}-720\alpha_{1,3}+540(M_1^2)^2-2340m^2M_1^2+2160mM_1^3-1080M_1^4+720M_2^1-240M_1^2M_1^3m-
120M_1^6-720M_2^2+600M_1^2M_1^3+1680M_1^2m^3-810(M_1^2)^2m-1920M_1^3m^2+1620M_1^4m+3600M_2^1m-720\Theta_2
-1260mM_1^2+720M_1^2+480M_1^3-240m^4M_1^2+320m^3M_1^3+360M_1^2M_2^1
+90M_1^2M_1^4+180(M_1^2)^2m^2-360M_1^4m^2-1440M_2^1m^2+288M_1^5m+1440\alpha_{1,2}m+40(M_1^3)^2-15(M_1^2)^3\Big]$.
\end{theorem}

\begin{proof}
The proof follows from Theorems \ref{66th1} and \ref{66th2}.
\end{proof}

If $T$ is an $n-$vertex tree, then $m=n-1$. Therefore by  previous theorem we have the following corollary:

\begin{corollary}\label{66th3}
Let $T$ be a tree on $n$ vertices. Then
$c_{n-6}(T)=\frac{1}{6!}\Big[(8(n-1)(8n^5-100n^4+410n^3-635n^2+355n-98)
-240M_1^2M_1^3n+1440\alpha_{1,2}n+288M_1^5n+9420M_1^2n-1170(M_1^2)^2n+2340M_1^4n+6480M_2^1n
+6960M_1^3n-1440M_2^1n^2+180(M_1^2)^2n^2-360M_1^4n^2+320M_1^3n^3-2880M_1^3n^2
-240M_1^2n^4+2640M_1^2n^3-8820M_1^2n^2+360M_1^2M_2^1+90M_1^2M_1^4+840M_1^2M_1^3
-15(M_1^2)^3+40(M_1^3)^2+1530(M_1^2)^2-720\Theta_2-120M_1^6-720M_2^2-3920M_1^3-3060M_1^4
-4320M_2^1-2280M_1^2-720\alpha_{1,2}-1008M_1^5-3600\alpha_{1,2}\Big]$.
\end{corollary}

\section{Laplacian Coefficients and the Number of Closed Walks}

Let $G$ be a graph. A \textit{walk} in $G$ is a sequence $W:v_{i_0}e_{i_1}v_{i_1}e_{i_2}v_{i_2}\ldots\,e_{i_k}v_{i_k}$ of vertices and edges of $G$ in such a way that for each $j$, $0\leq\,j\leq k-1$, $v_{i_j}$ and $v_{i_{j+1}}$ are end points of the edge $e_{i_{j+1}}$ in  $G$. The walk is said to be \textit{closed} if it begins and ends at the same vertex. The number of edges of a walk is called the \textit{length} of the walk.  The number of closed walk of a given length $k$, is denoted by $\mathcal{W}_k(G)$. It is easy to see that, in each graph $G$,  $\mathcal{W}_1(G)=0$, $\mathcal{W}_2(G)=2m(G)$ and  $\mathcal{W}_3(G)=6t(G)$.

The line graph of a given graph $G$ is another graph $L_1(G)$ that represents the adjacencies between edges of $G$.  This graph is constructed in this way: any  edge in $G$ will be a vertex in $L_1(G)$ and for two edges in $G$ with a common vertex, make an edge between their corresponding vertices in $L_1(G)$. For integer $k$, $k \geq2$, we define:
$L_k(G)=L_1(L_{k-1}(G))$ and $L_0(G)=G$.

\begin{theorem}\label{ttth0} $($ See {\rm \cite[Theorem 1.9]{book1}}$)$ Let $G$ be a graph with adjacency matrix $A$ and $k$ be a positive integer. Then $\Tr{A^k} =\mathcal{W}_k(G)$.
 \end{theorem}


The complete, star and cycle graphs on $n$ vertices are denoted by  $K_n$, $S_n$ and $C_n$, respectively. Suppose $V(S_5)=\{v_1,v_2,v_3,v_4,v_5\}$ and  $E(S_5)=\{v_1v_2,v_1v_3,v_1v_4,v_1v_5\}$. The graph $S_5^{2e}$ is constructed from the graph $S_5$ by adding two edges $v_2v_3$ and  $v_4v_5$.

\begin{lemma}\label{ttlm0}
Let  $G$ be a graph. The following hold:
 \begin{enumerate}
 \item $|\mathcal{S}_{P_3}(G)|=m(L_1(G))$.
 \item $M_1^2(G)=2\Big(m(G)+m(L_1(G))  \Big)$.
 \item $M_1^3(G)=2\big[m(G)+3m(L_1(G))+3t(L_1(G))-3t(G)\big]$.
 \end{enumerate}
 \end{lemma}

 \begin{proof}
\begin{enumerate}
\item Suppose that $e_1,e_2\in\,E(G)=V(L_1(G))$. By definition of line graph,  $e_1e_2\in\,E(L_1(G))$ if and only if $e_1$ and  $e_2$ have a common vertex. This proves that   $|\mathcal{S}_{P_3}(G)|=m(L_1(G))$.

\item  Choose the vertex  $v$ in $G$. The number of subgraphs of $G$ isomorphic to $P_3$ and middle vertex $v$ is equal to ${deg_G(v)\choose2}$. Hence $|\mathcal{S}_{P_3}(G)|$ $=$ $\sum_{v\in\,V(G)}{deg_G(v)\choose2}$ $=$ $\frac{1}{2} M_1^2(G)-m(G)$. By the case (1),  $M_1^2(G)=2(m(G)+m(L_1(G)))$, as desired.

 \item Suppose that $e_1,e_2,e_3\in\,E(G)=V(L_1(G))$. By definition, $L_1(G)[\{e_1,e_2,e_3\}]\cong\,C_3$ if and only if  $e_1$, $e_2$ and $e_3$ construct a cycle of length 3 or the star graph $S_4$. Therefore, $t(L_1(G))$ $=$ $\sum_{v\in\,V(G)}{deg_G(v)\choose3}+t(G)$ $=$ $\frac{1}{6}(M_1^3(G)-3M_1^2(G)+4m)+t(G)$. We now apply Lemma \ref{ttlm0}(2), to show that  $ t(L_1(G))=\frac{1}{6}(M_1^3(G)-2m(G)-6m(L_1(G))) + t(G)$.
\end{enumerate}
Hence the result.
 \end{proof}

Let $C_k:v_1v_2\ldots\,v_kv_1$ be the cycle graph on $k$ vertices. The graph $C_k[1^{l_1},2^{l_2},\ldots,k^{l_k}]$ is constructed from $C_k$ by adding $l_i$ pendant edges, $1\leq\,i\leq\,k$,  to the vertex $v_i$. For simplicity, if $l_i=0$, for some $i$, then we omit $i^{0}$ in our notation.

 \begin{lemma}\label{ttlm00}
Let  $G$ be a forest with $m(G)$ edges. Then the following hold:
{ \footnotesize \begingroup\allowdisplaybreaks\begin{align*}
 (i)\,M_1^4(G)&= \mathcal{W}_4(L_1(G))+2m(G)+12m(L_1(G))+36t(L_1(G)) -4m(L_2(G)),\\
 (ii)\,M_1^5(G)&= \mathcal{W}_5(L_1(G))+5M_1^4(G)-5M_1^3(G)-15M_1^2(G) +12m(G)-5\alpha_{1,2}(G)+30M_2^1(G),\\
 (iii)\,M_1^6(G)&=\mathcal{W}_6(L_1(G))-56m(L_1(G))+6M_1^5(G)-6\alpha_{1,3}(G) -6M_2^2(G)-60m(G) -9M_1^3(G)\\
 &-9M_1^4(G)+61M_1^2(G)-102M_2^1(G) -12m(L_2(G)) +42\alpha_{1,2}(G) -6\Theta_2(G)-6t(L_1(G)).
 \end{align*}\endgroup}
 \end{lemma}

\begin{proof}
Let  $H$ be an arbitrary graph.
\begin{enumerate}
\item[$(i)$]  It can be easily seen that
 \begin{align}
 \mathcal{W}_4(H)=2m(H)+4|\mathcal{S}_{P_3}(H)| +8|\mathcal{S}_{C_4}(H)|.\label{tteq3}
 \end{align}

Since $G$ is a forest, $|\mathcal{S}_{C_4}(L_1(G))|=3|\mathcal{S}_{K_4}(L_1(G))|=3\sum_{v\in\,V(G)}{deg_G(v)\choose4}$ $=$ $\frac{3}{24}(M_1^4(G)$ $-$ $6M_1^3(G)+11M_1^2(G)-12m(G))$, and by Lemma  \ref{ttlm0}(2,3),
 \begin{align}
 |\mathcal{S}_{C_4}(L_1(G))|&=\frac{3}{24}\left(M_1^4(G)-2m(G)-14m(L_1(G))-36t(L_1(G))\right).\label{tteq4}
 \end{align}

Also, by Lemma  \ref{ttlm0} (1),   $ |\mathcal{S}_{P_3}(L_1(G))|=m(L_2(G))$. We now apply Equations \eqref{tteq3} and  \eqref{tteq4} to deduce that $M_1^4(G)$ $=$ $\mathcal{W}_4(L_1(G))$ $+$ $2m(G)$ $+$ $12m(L_1(G))$ $+$ $36t(L_1(G))-4m(L_2(G))$, as desired.

\item[$(ii)$] By an easy calculation, one can see that
\begin{align}
\mathcal{W}_5(H)=30t(H)+10 |\mathcal{S}_{C_5}(H)|+10 |\mathcal{S}_{C_3[1^1]}(L_1(G))|\label{tteq31}
\end{align}

Since $G$ is a forest, {\small$|\mathcal{S}_{C_5}(L_1(G))|$ $=$ $12|\mathcal{S}_{K_5}(L_1(G))|$ $=$ $12\sum_{v\in\,V(G)}{deg_G(v)\choose5}$ $=$ $\frac{12}{120}(M_1^5(G)$ $-$ $10M_1^4(G)+35M_1^3(G)$}
 $-50M_1^2(G)+48m(G))$ and $|\mathcal{S}_{C_3[1^1]}(L_1(G))|$ $=$ $\sum_{uv\in E(G)} \Big[$ ${deg_G(u)-1\choose2}$ $(deg_G(u)+deg_G(v)-4)$
 $+$ ${deg_G(v)-1\choose2}$ $(deg_G(u)+deg_G(v)-4)\Big]$
$=\frac{1}{2}M_1^4(G)+\frac{1}{2}\alpha_{1,2}(G)-\frac{7}{2}M_1^3(G)-3M_2^1(G)$
 $+8M_1^2(G)-8m(G)$.
Now the result follows from Equation  \eqref{tteq31}.

\item[$(iii)$] By some easy calculations, one can see that
 \begingroup\allowdisplaybreaks \begin{align}
 \mathcal{W}_6(H)&=2m(H)+12|\mathcal{S}_{P_3}(H)|+6|\mathcal{S}_{P_4}(H)|+12|\mathcal{S}_{S_4}(H)|+24t(H)+ 48|\mathcal{S}_{C_4}(H)| \nonumber\\
 &+36|\mathcal{S}_{K_4-e}(H)|+12 |\mathcal{S}_{C_4[1^1]}(H)|+12|\mathcal{S}_{C_6}(H)|+24|\mathcal{S}_{S_4^{2e}}(H)|.
 \label{tteq32}
 \end{align}\endgroup

We now assume that $x=uvw\in\,\mathcal{S}_{P_3}(G)=E(L_1(G))$. Then the number of paths constructed from three edges in $L_1(G)$ with  $x$ as its middle edge can be computed by  $(deg_G(u)+deg_G(v)-3)(deg_G(v)+deg_G(w)-3)-t_{x}(L_1(G)),$ where $ t_x(L_1(G))$ denotes the number of triangles constructed on the edge $x$ of $L_1(G)$. Therefore, $|\mathcal{S}_{P_4}(L_1(G))|$ $=$ $\sum_{uvw\in\mathcal{S}_{P_3}(G)}[(deg_G(u)+deg_G(v)-3)(deg_G(v)+deg_G(w)-3)-t_{x}(L_1(G))]$ $=$  $\sum_{uv\in\,E(G)}deg_G(u)deg_G(v)(deg_G(u)+deg_G(v)-2)$ $-$ $3\sum_{uv\in\,E(G)}[deg_G(u)(deg_G(v)-1)+(deg_G(u)-1)deg_G(v)]$ $+$ $\Theta_2(G)$ $-$ $3t(L_1(G))$ $+$ $9m(L_1(G))$ $+$ $\sum_{v\in\,V(G)}{deg_G(v)\choose2}$ $(deg_G(v)^2-6deg_G(v))$. Now we simplify the last summation to deduce that
\begingroup\allowdisplaybreaks \begin{align}
|\mathcal{S}_{P_4}(L_1(G))|&=\Theta_2(G)-3t(L_1(G))+9m(L_1(G))+\alpha_{1,2}(G)-8M_2^1(G)\nonumber\\
&+6M_1^2(G)+\frac{1}{2}M_1^4(G)-\frac{7}{2}M_1^3(G).\label{tttteq0}
 \end{align}\endgroup

Suppose that $e=uv\in\,E(G)=V(L_1(G))$. The number of stars isomorphic to $S_4$ in  $L_1(G)$ with  $e$ as its center is computed by ${deg_G(u)+deg_G(v)-2\choose3}$ and so
{\small \begin{align}
 |\mathcal{S}_{S_4}(L_1(G))|&=\frac{1}{6}M_1^4(G)+\frac{1}{2}\alpha_{1,2}(G)-\frac{3}{2}M_1^3(G)-3M_2^1(G)+\frac{13}{3}M_1^2(G)-4m(G).\label{tttteq1}
 \end{align}}
By the proof of Case (1), we have
\begingroup\allowdisplaybreaks \begin{align}
 |\mathcal{S}_{C_4}(L_1(G))|&=\frac{3}{24}\Big(M_1^4(G)-6M_1^3(G)+11M_1^2(G) -12m(G)\Big).\label{tttteq2}
 \end{align}\endgroup

Furthermore, it can be seen that
 \begin{align}
 |\mathcal{S}_{K_4-e}(L_1(G))|=2|\mathcal{S}_{C_4}(L_1(G))|.\label{tttteq3}
 \end{align}

On the other hand, by definition of complete graphs,
 \begin{align}
 |\mathcal{S}_{C_n}(K_n)|=\frac{1}{2}(n-1)!.\label{tttteq4}
 \end{align}

Note that four edges in $G$ gives an induced subgraph of  $L_1(G)$ isomorphic to $K_4$ if and only if those edges has a common vertex. Thus,
\begingroup\allowdisplaybreaks  \begin{align}
 |\mathcal{S}_{C_4[1^1]}(L_1(G))|&=3\sum_{e=uv\in\,E(G)}\Big[{deg_G(u)-1\choose3}(deg_G(u)+deg_G(v)-5)\nonumber\\
 &+{deg_G(v)-1\choose3}(deg_G(u)+deg_G(v)-5)\Big]\\
 &=\frac{1}{2}M_1^5(G)+\frac{1}{2}\alpha_{1,3}(G)-\frac{11}{2}M_1^4(G)-3\alpha_{1,2}(G)\nonumber\\
 &+\frac{41}{2}M_1^3(G) -\frac{67}{2}M_1^2(G)+11M_2^1(G)+30m(G).\label{tttteq5}
 \end{align}\endgroup

Furthermore, six edges of $G$ gives a cycle in $L_1(G)$ if and only if those edges have a common vertex. We now apply Equation \eqref{tttteq4} to deduce that
\begingroup\allowdisplaybreaks \begin{align}
 |\mathcal{S}_{C_6}(L_1(G))|&=60|\mathcal{S}_{K_6}(L_1(G))|=60\sum_{v\in\,V(G)}{deg_G(v)\choose6}\nonumber\\
 &=\frac{1}{12}M_1^6(G)-\frac{5}{4}M_1^5(G)+\frac{85}{12}M_1^4(G)-\frac{75}{4}M_1^3(G)\nonumber\\
 &+\frac{137}{6}M_1^2(G)-20m(G).\label{tttteq6}
 \end{align}\endgroup

Suppose $f=uv\in\,E(G)=V(L_1(G))$. The number of subgraphs of  $L_1(G)$ isomorphic to  $S_5^{2e}$ with the property that  $f$ is a vertex of degree $4$ can be obtained from  ${deg_G(u)-1\choose2}{deg_G(v)-1\choose2}+3{deg_G(u)-1\choose4}+3{deg_G(v)-1\choose4}$. Therefore,
  \begingroup\allowdisplaybreaks\begin{align}
 |\mathcal{S}_{S_5^{2e}}(L_1(G))|&=\frac{1}{4}M_2^2(G)-\frac{3}{4}\alpha_{1,2}(G)+\frac{39}{8}M_1^3(G) +\frac{9}{4}M_2^1(G)\nonumber\\
 &-\frac{31}{4}M_1^2(G)+7m(G)+\frac{1}{8}M_1^5(G)-\frac{5}{4}M_1^4(G).\label{tttteq7}
 \end{align}\endgroup

We now apply Lemma \ref{ttlm0} and Equations  \eqref{tteq32}, \eqref{tttteq0}, \eqref{tttteq1}, \eqref{tttteq3}, \eqref{tttteq4}, \eqref{tttteq5}, \eqref{tttteq6} and \eqref{tttteq7} to complete the proof of this case.
\end{enumerate}
Hence the result.
 \end{proof}

\begin{lemma}\label{ttlm4}
Suppose that $G$ is a graph. Then the following equalities hold:
 \begingroup\allowdisplaybreaks\begin{align*}
 M_2^1(G)&=\frac{1}{2} M_1^2(L_1(G))-\frac{1}{2} M_1^3(G)+2M_1^2(G)-2m(G),\\
EM_2(G)&=\frac{1}{2} M_1^2(L_2(G))-\frac{1}{2} M_1^3(L_1(G))+2M_1^2(L_1(G))-2m(L_1(G)),\\
\alpha_{1,2}(G)&=\frac{1}{3}M_1^3(L_1(G))-\frac{1}{3}M_1^4(G)+2M_1^3(G)+4M_2^1(G) -4M_1^2(G)+\frac{8}{3}m(G),\\
\Theta_2(G)&=\frac{1}{2}M_1^2(L_2(G))-\frac{5}{6}M_1^3(L_1(G))+2M_1^2(L_1(G))-2m(L_1(G))\\
&-\frac{1}{6}M_1^4(G)+\frac{1}{2}M_1^3(G)+2M_2^1(G)-2M_1^2(G)+\frac{4}{3}m(G).
\end{align*}\endgroup
\end{lemma}

\begin{proof}
Suppose that  $e=uv\in\,E(G)$. By definition of line graph, $deg_{L_1(G)}(e)$ $=$ $deg_G(u)$ $+$ $deg_G(v)$ $-$ $2$. Hence $M_1^2(L_1(G))$ $=$ $\sum_{uv\in\,E(G)}(deg_G(u)+deg_G(v)-2)^2$ $=$ $M_1^3(G)+2M_2^1(G)-4M_1^2(G)+4m(G)$ which completes the proof of the first equality. The second equality follows from $EM_2(G)=M_2(L_1(G))$ and the first equality.

Next, we prove the third equality. We have  $M_1^3(L_1(G))$ $=$ $\sum_{uv\in\,E(G)}(deg_G(u)+deg_G(v)-2)^3$ $=$ $3\alpha_{1,2}(G)+M_1^4(G)-6M_1^3(G)-12M_2^1(G)$ $+$ $12M_1^2(G)-8m(G)$, as desired. Finally, by Lemma \ref{6lmm13},
$\Theta_2(G)$ $=$ $EM_2(G)-\alpha_{1,2}+6M_2^1(G)-\frac{1}{2}M_1^4(G)+\frac{5}{2}M_1^3(G)$ $-6M_1^2(G)+4m(G)$. Now, the fourth equality follows from the above three equalities.
\end{proof}

We are now ready to prove the main result of this section.

\begin{theorem}
Let $G$ be a graph, $A=A(G)$ and  $A_1=A(L_1(G))$. Then the following hold:

 \begin{enumerate}
\item $c_{n-1}(G)=\Tr A^2$ and $c_{n-2}(G) = \frac{1}{2}\Big[\prod_{i=0}^1\big(\Tr A^2-2i)-\Tr A_1^2\Big]$.

\item  $c_{n-3}(G)$ $=$ $\frac{1}{3!}\Big[\prod_{i=0}^2\big(\Tr A^2-2i)-3\Tr A^2\Tr A_1^2+\Tr(12A_1^2+2A_1^3)-4\Tr A^3 \Big]$.

\item $c_{n-4}(G)$ $=$ $\frac{1}{4!}\Big[\prod_{i=0}^3\big(\Tr A^2-2i)-6(\Tr A^2)^2\Tr A_1^2
 +\Tr A^2\Tr(60A_1^2+8A_1^3) -\Tr(144A_1^2+48A_1^3+6A_1^4)+3(\Tr A_1^2)^2 \Big]$, when $G$ is a forest.

\item $c_{n-5}(G)$ $=$ $\frac{1}{5!}\Big[\prod_{i=0}^4\big(\Tr A^2-2i)-10(\Tr A^2)^3\Tr A_1^2
 +(\Tr A^2)^2\Tr(180A_1^2+20A_1^3)-\Tr A^2\Tr(1040A_1^2+280A_1^3+30A_1^4)+15\Tr A^2(\Tr A_1^2)^2-\Tr A_1^2\Tr(120A_1^2+20A_1^3)+\Tr(1920A_1^2$ $+$ $960A_1^3+240A_1^4+24A_1^5)\Big]$, when $G$ is a forest.

\item $c_{n-6}(G)$ $=$ $\frac{1}{6!}\Big[\prod_{i=0}^5\big(\Tr A^2-2i)-15(\Tr A^2)^4\Tr A_1^2
 +(\Tr A^2)^3\Tr(420A_1^2+40A_1^3)$ $-$ $(\Tr A^2)^2\Tr(4260A_1^2$ $+$ $960A_1^3$ $+$ $90A_1^4)$ $+$ $45(\Tr A^2\Tr A_1^2)^2$ $-$ $810\Tr A^2(\Tr A_1^2)^2$ $+$ $\Tr A^2\Tr$ $(18480A_1^2$ $+$ $7520A_1^3$ $+$ $1620A_1^4$ $+$ $144A_1^5)$ $-$ $15(\Tr A_1^2)^3$ $+$ $3600(\Tr A_1^2)^2$ $-$ $28800\Tr A_1^2$ $+$ $\Tr A_1^2\Tr(1200A_1^3$ $+$ $90A_1^4)$ $+$ $40(\Tr A_1^3)^2$ $-$ $120\Tr A^2\Tr A_1^2\Tr A_1^3$ $-$ $\Tr(19200A_1^3$ $-$ $7200A_1^4$ $-$ $1440A_1^5$ $-$ $120A_1^6) \Big]$, when $G$ is a forest.
\end{enumerate}
\end{theorem}

\begin{proof} The proof follows from Theorems \ref{66th2}, \ref{ttth0} and Lemmas \ref{ttlm0},  \ref{ttlm00},
\ref{ttlm4}.
\end{proof}

\section{Applications}

This section aims is to apply our results in Section 3 for computing the Laplacian coefficients $c_{n-k}$, $k =2,3,4,5,6$, when $G$ is a certain tree.
Let $G$ be a graph. The number of edges connecting vertices of degree $i$ and $j$ in a graph $G$ is denoted by $m_{i,j}$. Let
$\mathcal{P}_3^{i,j}=|\{uv w \in \mathcal{S}_{P_3}(G)\mid\{deg_G(u),deg_G(w)\}=\{i,j\}\}|$.
 As \cite{2-2}, we first assume that $T (k, t)$ be a rooted tree with degree
sequence $k, k,\ldots, k, 1, 1,\ldots, 1$ and $t$ is the distance between the center and any pendant
vertex, Figure \ref{sh4}. By definition of $T (k, t)$ we have
$n\big(T (k, t)\big)$ $=$ $\frac{k(k-1)^t-2}{k-2}$,
$m_{1,k}\big(T (k, t)\big)$ $=$ $k(k-1)^{t-1}$, $m_{k,k}\big(T (k, t)\big)$ $=$ $n\big(T (k, t)\big)-1-m_{1,k}\big(T (k, t)\big)$,
$\mathcal{P}_3^{1,k}$ $=$ $m_{1,k}\big(T (k, t)\big)$, $\mathcal{P}_3^{1,1}$ $=$ $k(k-1)^{t-2}{k-1\choose 2}$,
$\mathcal{P}_3^{k,k}$ $=$ $\frac{1}{2}M_1\big(T (k, t)\big)-n\big(T (k, t)\big)+1-\mathcal{P}_3^{1,1}-\mathcal{P}_3^{1,k}$.
Therefore by Lemma \ref{l1}, Equations \eqref{oeq1}, \eqref{oeq2} and  Theorem \ref{66th3} and the others  simple calculations  we have:
 {\footnotesize \begingroup\allowdisplaybreaks\begin{align*}
&c_{n-1}(T(3,t))=6\times2^t-6,\\
&c_{n-2}(T(3,t))=18\times2^{2t}-\frac{93}{2}2^t+30,\\
&c_{n-3}(T(3,t)) =36\times2^{3t}-171\times2^{2t}
+272\times2^t-144,\\
&c_{n-4}(T(3,t))=54\times2^{4t}-405\times2^{3t}+\frac{9177}{8}2^{2t}-\frac{5799}{4}2^t+687 ,\\
&c_{n-5}(T(3,t))=\frac{324}{5}2^{5t}-702\times2^{4t}+\frac{12267}{4}2^{3t}-\frac{26967}{4}2^{2t}+\frac{74427}{10}2^t
-3294,\\
& c_{n-6}(T(3,t))={\frac {324\,}{5}{2}^{6\,t}}-{\frac {4779\,}{5}{2}^{5\,t}}+{\frac {23697\,}{4}{2}^{4\,t}}-{
\frac {315711\,}{16}{2}^{3\,t}}+{\frac {1488293\,}{40}{2}^{2\,t}}-{\frac {376247\,}{10}{2}^{t}}+15932,\\
&c_{n-1}(T(4,t))=4\times3^t-4,\\
&c_{n-2}(T(4,t))=8\times3^{2t}-24\times3^t+18 ,\\
&c_{n-3}(T(4,t))=\frac{32}{3}3^{3t}-64\times3^{2t} +\frac{392}{3}3^t-88,\\
&c_{n-4}(T(4,t))=\frac{32}{3}3^{4t}-\frac{320}{3}3^{3t}+\frac{1232}{3}3^{2t}-\frac{2132}{3}3^t+457 ,\\
&c_{n-5}(T(4,t))=\frac{128}{15}3^{5t}-128\times3^{4t}+\frac{2368}{3}3^{3t}-2480\times3^{2t}+\frac{19644}{5}3^t-2484,\\
&c_{n-6}(T(4,t))={\frac {256\,}{45}{3}^{6\,t}}-{\frac {1792\,}{15}{3}^{5\,t}}+{\frac {9664\,}{9}{3}^{4\,t}}-{\frac
{15776\,}{3}{3}^{3\,t}}+{\frac {661864\,}{45}{3}^{2\,t}}-{\frac {110756\,}{5}{3}^{t}}+13990.
\end{align*}\endgroup
\begin{table}[htp]
\centering\caption{  The Laplacian coefficients $c_{n-x}\big(T(3,t)\big)$, where $x,t\in\{2,3,4,5,6\}$. }
 \begingroup\allowdisplaybreaks\begin{tabular}{|c|c|c|c|c|c|c|}
\hline
\diagbox{x}{t} &2&3&4&5&6\\
\hline
2 &132&810& 3894& 16974&70782\\
3& 512&9520&107888&1013104&8754032 \\
4 &1146&76329&2151219&44481015&804407871 \\
5 & 1524&442926&32892762&1532049426&58577653506 \\
6 &  1196&1926456&401303300& 43109506572&3521109479132 \\
\hline
\end{tabular}\endgroup
\end{table}

 \begingroup\allowdisplaybreaks\begin{table}[h!]
\centering\caption{  The Laplacian coefficients $c_{n-x}\big(T(4,t)\big)$, where $x,t\in\{2,3,4,5,6\}$. }
\scalebox{.9}{\begin{tabular}{|c|c|c|c|c|c|c|}
\hline
\diagbox{x}{t} &2&3&4&5&6\\
\hline
2 &450&5202&50562& 466578& 4234050\\
3&3680&166736&5259296&149307536& 4098568160 \\
4 & 19549&3849829&405115261&35685894085&2971474597789 \\
5 & 71496&   68251680&24647441832&6795068311872&1721091168665352 \\
6 &  186394& 967057330&1233678403066&1073738466435154& 829575812820551386\\
\hline
\end{tabular}}
\end{table}\endgroup

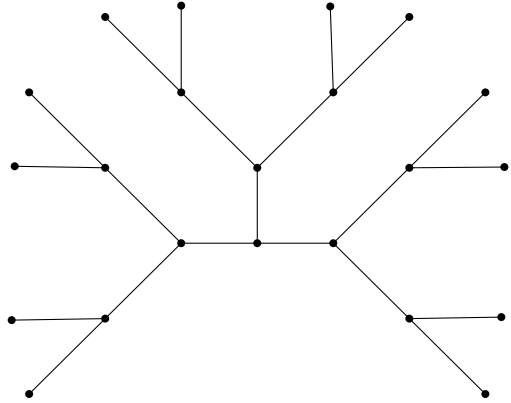
\begin{figure}[htp]
\begin{center}
\begin{tikzpicture}[line cap=round,line join=round,>=triangle 45,x=1.0cm,y=1.0cm]
\clip(-7,-2.1) rectangle (7,3.5);
\draw (0,0)-- (0,1);
\draw (0,0)-- (-1,0);
\draw (0,0)-- (1,0);
\draw (0,1)-- (-1,2);
\draw (0,1)-- (1,2);
\draw (1,0)-- (2,1);
\draw (1,0)-- (2,-1);
\draw (-1,0)-- (-2,1);
\draw (-1,0)-- (-2,-1);
\draw (-2,1)-- (-3,2);
\draw (-2,1)-- (-3.19,1.02);
\draw (-2,-1)-- (-3,-2);
\draw (-2,-1)-- (-3.23,-1.02);
\draw (-1,2)-- (-2,3);
\draw (-1,2)-- (-1,3.15);
\draw (1,2)-- (2,3);
\draw (1,2)-- (0.96,3.14);
\draw (2,1)-- (3,2);
\draw (2,1)-- (3.25,1.01);
\draw (2,-1)-- (3,-2);
\draw (2,-1)-- (3.21,-0.98);
\draw (-0.53,-1.63) node[anchor=north west] {};
\begin{scriptsize}
\fill [color=black] (0,0) circle (1.5pt);
\fill [color=black] (0,1) circle (1.5pt);
\fill [color=black] (-1,0) circle (1.5pt);
\fill [color=black] (1,0) circle (1.5pt);
\fill [color=black] (-1,2) circle (1.5pt);
\fill [color=black] (1,2) circle (1.5pt);
\fill [color=black] (2,1) circle (1.5pt);
\fill [color=black] (2,-1) circle (1.5pt);
\fill [color=black] (-2,1) circle (1.5pt);
\fill [color=black] (-2,-1) circle (1.5pt);
\fill [color=black] (-3,2) circle (1.5pt);
\fill [color=black] (-3.19,1.02) circle (1.5pt);
\fill [color=black] (-3,-2) circle (1.5pt);
\fill [color=black] (-3.23,-1.02) circle (1.5pt);
\fill [color=black] (-2,3) circle (1.5pt);
\fill [color=black] (-1,3.15) circle (1.5pt);
\fill [color=black] (2,3) circle (1.5pt);
\fill [color=black] (0.96,3.14) circle (1.5pt);
\fill [color=black] (3,2) circle (1.5pt);
\fill [color=black] (3.25,1.01) circle (1.5pt);
\fill [color=black] (3,-2) circle (1.5pt);
\fill [color=black] (3.21,-0.98) circle (1.5pt);
\end{scriptsize}
\end{tikzpicture}
\caption{ The rooted tree $T(3,3)$. \label{sh4}}
\end{center}
\end{figure}

\section{Concluding Remarks}
In this paper, an  exact formula for the coefficient $c_{n-6}$ in the Laplacian polynomial of a forest is given. As a consequence of this work and our earlier papers, the Laplacian coefficients $c_{n-k}$ of a forest $F$, $1 \leq k \leq 6$,  are computed in terms of the number of closed walks in $F$ and its line graph. We end this paper by the following conjecture:

\begin{con}
Let $G$ be a graph, $A=A(G)$ and  $A_1=A(L_1(G))$. Then
\begin{enumerate}
\item $c_{n-4}(G)$ $\leq$ $\frac{1}{4!}\Big[\prod_{i=0}^3\big(\Tr A^2-2i)-6(\Tr A^2)^2\Tr A_1^2
 +\Tr A^2\Tr(60A_1^2+8A_1^3) -\Tr(144A_1^2+48A_1^3+6A_1^4)+3(\Tr A_1^2)^2 \Big]$, the equality holds if and only if the girth of $G$ is not $3$.
 
\item $c_{n-5}(G)$ $\leq$ $\frac{1}{5!}\Big[\prod_{i=0}^4\big(\Tr A^2-2i)-10(\Tr A^2)^3\Tr A_1^2
 +(\Tr A^2)^2\Tr(180A_1^2+20A_1^3)-\Tr A^2\Tr(1040A_1^2+280A_1^3+30A_1^4)+15\Tr A^2(\Tr A_1^2)^2-\Tr A_1^2\Tr(120A_1^2+20A_1^3)+\Tr(1920A_1^2+960A_1^3+240A_1^4+24A_1^5)\Big]$, the equality holds if and only if
  the girth of $G$ is not $3$ or $5$.

\item $c_{n-6}(G)$ $\leq$ $\frac{1}{6!}\Big[\prod_{i=0}^5\big(\Tr A^2-2i)-15(\Tr A^2)^4\Tr A_1^2
 +(\Tr A^2)^3\Tr(420A_1^2+40A_1^3)$ $-$ $(\Tr A^2)^2\Tr(4260A_1^2$ $+$ $960A_1^3$ $+$ $90A_1^4)$ $+$ $45(\Tr A^2\Tr A_1^2)^2$ $-$ $810\Tr A^2(\Tr A_1^2)^2$ $+$ $\Tr A^2\Tr(18480A_1^2$ $+$ $7520A_1^3$ $+$ $1620A_1^4$ $+$ $144A_1^5)$ $-$ $15(\Tr A_1^2)^3$ $+$ $3600(\Tr A_1^2)^2$ $-$ $28800\Tr A_1^2$ $+$ $\Tr A_1^2\Tr(1200A_1^3$ $+$ $90A_1^4)$ $+$ $40(\Tr A_1^3)^2$ $-$ $120\Tr A^2\Tr A_1^2\Tr A_1^3$ $-$ $\Tr(19200A_1^3$ $-$ $7200A_1^4$ $-$ $1440A_1^5$ $-$ $120A_1^6) \Big]$, the equality holds if and only if
  the girth of $G$ is not $3$ or $5$.
\end{enumerate}
\end{con}

\vskip 3mm

\noindent\textbf{Acknowledgments:} The authors supported by the University of Kashan under grant no. 985968/1.

\bigskip

\vskip 0.4 true cm


\end{document}